\documentclass[12pt, reqno]{amsart}
\usepackage{amsmath, amsthm, amscd, amsfonts, amssymb}
\usepackage[mathscr]{eucal}
\usepackage[bookmarksnumbered, colorlinks, plainpages]{hyperref}
\hypersetup{colorlinks=true,linkcolor=red, anchorcolor=green, citecolor=cyan, urlcolor=red, filecolor=magenta, pdftoolbar=true}

\usepackage[pdftex]{graphicx}

\newtheorem{theorem}{Theorem}

\newtheorem{corollary}[theorem]{Corollary}
\theoremstyle{remark}
\newtheorem{remark}{Remark}
\newtheorem{example}{Example}

\DeclareMathOperator{\co}{co}
\DeclareMathOperator{\Real}{Re}

\begin{document}
\setcounter{page}{1}

\title[An estimate of approximation]{An estimate of approximation\\ of a matrix-valued function\\ by an interpolation polynomial}

\author{V.~G. Kurbatov}
 \address{Department of Mathematical Physics,
Voronezh State University\\ 1, Universi\-tet\-skaya Square, Voronezh 394018, Russia}
\email{\textcolor[rgb]{0.00,0.00,0.84}{kv51@inbox.ru}}

\author{I.~V. Kurbatova}
 \address{Department of Software Development and Information Systems Administration,
Vo\-ro\-nezh State University\\ 1, Universitetskaya Square, Voronezh 394018, Russia}
\email{\textcolor[rgb]{0.00,0.00,0.84}{la\_soleil@bk.ru}}

\subjclass{Primary 65F60; Secondary 97N50}

\keywords{function of a matrix, interpolation polynomial, estimate}

\date{\today}

\begin{abstract}
Let $A$ be a square complex matrix, $z_1$, \dots, $z_{n}\in\mathbb C$~be {\rm(}possibly repetitive{\rm)} points of interpolation, $f$ be a function analytic in a neighborhood of the convex hull of the union of the spectrum of $A$ and the points $z_1$, \dots, $z_{n}$, and $p$~be the interpolation polynomial of $f$ constructed by the points $z_1$, \dots, $z_{n}$. It is proved that under these assumptions
\begin{equation*}
\lVert f(A)-p(A)\rVert\le\frac1{n!}
\max_{\substack{t\in[0,1]\\\mu\in\co\{z_1,z_{2},\dots,z_{n}\}}}\bigl\lVert\Omega(A)f^{{(n)}}
\bigl((1-t)\mu\mathbf1+tA\bigr)\bigr\rVert,
\end{equation*}
where $\Omega(z)=\prod_{k=1}^n(z-z_k)$ and the symbol $\co$ means the convex hull.
\end{abstract}

\maketitle

\section*{Introduction}\label{s:Introduction}
An approximate calculation of analytic functions of matrices~\cite{Frommer-Simoncini08a,Higham08} arises in many applications.
One of the often used methods for the approximate calculation of a function $f$ of a large matrix $A$ is the replacement of $f$ by its polynomial approximation $p$. For the approximation by the Taylor polynomial, it is known a good estimate of accuracy~\cite{Mathias93a}, see Corollary~\ref{c:Mathias}
.
In this paper, we propose an estimate of the norm $\lVert f(A)-p(A)\rVert$, which is a generalization of the estimate from~\cite{Mathias93a} for the case when $p$ is an interpolation polynomial of $f$. This estimate may help to choose an interpolation polynomial for an approximate calculation of a matrix function in an optimal way. Our estimate can be considered as a matrix analogue of the estimate~\cite[Theorem 3.1.1]{Davis_PhJ}
\begin{equation*}
|f(z)-p(z)|\le\frac{1}{n!}|\Omega(z)|\max_{\lambda\in\co\{z_1,z_{2},\dots,z_{n},\,z\}}|f^{(n)}(\lambda)|,
\end{equation*}
where
\begin{equation*}
\Omega(z)=\prod_{k=1}^n(z-z_k),
\end{equation*}
for the difference between an analytic function $f$ and its interpolation polynomial $p$ with respect to the points of interpolation $z_1$, \dots, $z_{n}$ provided that $f$ is analytic in a neighborhood of the convex hull of the points $z_1$, \dots, $z_{n}$ and $z$.

It is known many estimates of $\lVert f(A)\rVert$, see, e.g.,~\cite{Crouzeix04,Gil-LMA-93,Gil-MPAG-08,Gil-ASM-14,Greenbaum04,Kagstrom77,Van_Loan75:SME,Van_Loan77,Young81}. All of them can be equivalently written as estimates of $\lVert f(A)-p(A)\rVert$, see, e.g.,~\cite[Theorem 11.2.2]{Golub-Van_Loan96:eng}.
The difference between these estimates and the proposed one (Theorem~\ref{t:Mathias2}) is that the latter is adapted for the approximation by an interpolation polynomial.

In Section~\ref{s:estimate}, we prove our estimate (Theorem~\ref{t:Mathias2}) and describe some of its variants for the case of the matrix exponent. In Section~\ref{s:experiment}, we give a numerical application.

\section{The estimate}\label{s:estimate}
\begin{theorem}\label{t:Mathias2}
Let $A$ be a square complex matrix, $z_1$, \dots, $z_{n}\in\mathbb C$~be arbitrary {\rm(}possibly repetitive{\rm)} points of interpolation, $f$ be an analytic function defined in a neighborhood of the convex hull of the union of the spectrum of $A$ and the points $z_1$, \dots, $z_{n}$, and $p$~be the interpolation polynomial of $f$ constructed by the points $z_1$, \dots, $z_{n}$ {\rm(}taking into account their multiplicities{\rm)}. Then {\rm(}for any norm on the space of matrices{\rm)}
\begin{equation*}
\lVert f(A)-p(A)\rVert\le\frac1{n!}
\max_{\substack{t\in[0,1]\\\mu\in\co\{z_1,z_{2},\dots,z_{n}\}}}\bigl\lVert\Omega(A)f^{{(n)}}
\bigl((1-t)\mu\mathbf1+tA\bigr)\bigr\rVert,
\end{equation*}
where $\mathbf1$ is the identity matrix, the symbol $\co$ means the convex hull, and
\begin{equation*}
\Omega(z)=\prod_{k=1}^n(z-z_k).
\end{equation*}
\end{theorem}

\begin{proof}
It is well-known (see, e.g.,~\cite[Theorem 3.4.1]{Davis_PhJ} or~\cite[formula (52)]{Gelfond:eng1}) that
\begin{equation*}
f(z)-p(z)=\Omega(z)f[z_1,z_{2},\dots,z_{n},z],
\end{equation*}
where $f[z_1,z_{2},\dots,z_{n},z]$ is the divided difference~\cite{Davis_PhJ,Gelfond:eng1,Jordan}.
On the other hand, by~\cite[formula (47)]{Gelfond:eng1}, we have
\begin{multline}\label{e:formula (47)}
f[z_1,z_{2},\dots,z_{n},z]=\int_0^1\int_0^{t_1}\dots\int_0^{t_{n-1}}f^{(n)}
\bigl(z_1+(z_2-z_1)t_1+\dots\\
+(z_{n}-z_{n-1})t_{n-1}+(z-z_{n})t_{n}\bigr)\,dt_{n}dt_{n-1}\dots dt_1.
\end{multline}
Or
\begin{multline*}
f[z_1,z_{2},\dots,z_{n},z]=\int_0^1\int_0^{t_1}\dots\int_0^{t_{n-1}}f^{{(n)}}
\bigl((1-t_1)z_1+(t_1-t_2)z_2+\dots\\
+(t_{n-1}-t_{n})z_{n}+t_{n}z\bigr)\,dt_{n}dt_{n-1}\dots dt_1.
\end{multline*}
Clearly, the complex numbers
\begin{equation*}
(1-t_1)z_1+(t_1-t_2)z_2+\dots+(t_{n-1}-t_{n})z_{n}+t_{n}z
\end{equation*}
form the convex hull of the set $\{z_1,z_{2},\dots,z_{n},z\}$ when $t_1,\dots,t_n$ run through the set specified by the inequalities $0\le t_{n}\le\dots\le t_1\le1$. Thus, considering integral~\eqref{e:formula (47)}, we use the fact that $f^{(n)}$ is defined on the convex hull of the points $z_1$, \dots, $z_{n}$, and $z$.

Substituting $A$ for $z$ into the previous formulas (thus, we assume that any point of the spectrum of $A$ can be taken as $z$, which can be done, since $f$ is analytic in a neighborhood of the convex hull of the union of the spectrum of $A$ and the points $z_1$, \dots, $z_{n}$), we obtain
\begin{multline*}
f(A)-p(A)=\Omega(A)\int_0^1\int_0^{t_1}\dots\int_0^{t_{n-1}}f^{{(n)}}
\bigl((1-t_1)z_1\mathbf1+(t_1-t_2)z_2\mathbf1+\dots\\
+(t_{n-1}-t_{n})z_{n}\mathbf1+t_{n}A\bigr)\,dt_{n}dt_{n-1}\dots dt_1.
\end{multline*}

Let $\xi$~be a linear functional on the space of matrices (equipped by an arbitrary norm) such that $\lVert\xi\rVert=1$ and
\begin{equation*}
\lVert f(A)-p(A)\rVert=\xi\bigl(f(A)-p(A)\bigr).
\end{equation*}
Such a functional exists by the Hahn-Banach theorem~\cite[Theorem 2.7.4]{Hille-Phillips:eng}. Then we have the estimate
\begin{equation}\label{e:last est}
\begin{split}
\lVert f(A)-p(A)\rVert
&=\xi\biggl(\int_0^1\int_0^{t_1}\dots\int_0^{t_{n-1}}\Omega(A)f^{{(n)}}
\bigl((1-t_1)z_1\mathbf1+\dots\\
&\phantom{fghfg}+(t_{n-1}-t_{n})z_{n}\mathbf1+t_{n}A\bigr)\,dt_{n}dt_{n-1}\dots dt_1\biggr)\\
&\le\biggl\lVert \int_0^1\int_0^{t_1}\dots\int_0^{t_{n-1}}\Omega(A)f^{{(n)}}
\bigl((1-t_1)z_1\mathbf1+\dots\\
&\phantom{fghfg}+(t_{n-1}-t_{n})z_{n}\mathbf1+t_{n}A\bigr)\,dt_{n}dt_{n-1}\dots dt_1\biggr\rVert\\
&\le\int_0^1\int_0^{t_1}\dots\int_0^{t_{n-1}}\bigl\lVert\Omega(A)f^{{(n)}}
\bigl((1-t_1)z_1\mathbf1+\dots\\
&\phantom{fghfg}+(t_{n-1}-t_{n})z_{n}\mathbf1+t_{n}A\bigr)\bigr\rVert\,dt_{n}dt_{n-1}\dots dt_1\\
&\le\int_0^1\int_0^{t_1}\dots\int_0^{t_{n-1}}
\max_{0\le t_{n}\le\dots\le t_1\le1}\bigl\lVert\Omega(A)f^{{(n)}}
\bigl((1-t_1)z_1\mathbf1+\dots\\
&\phantom{fghfg}+(t_{n-1}-t_{n})z_{n}\mathbf1+t_{n}A\bigr)\bigr\rVert\,dt_{n}dt_{n-1}\dots dt_1.
\end{split}
\end{equation}

We observe that the complex numbers
\begin{equation*}
\frac1{1-t_{n}}\bigl((1-t_1)z_1+(t_1-t_2)z_2+\dots+(t_{n-1}-t_{n})z_{n}\bigr)
\end{equation*}
form the convex hull of $\{z_1,z_{2},\dots,z_{n}\}$ when $t_1,\dots,t_n$ run through the set specified by the inequalities $0\le t_{n}\le\dots\le t_1\le1$. Besides,
\begin{equation*}
\int_0^1\int_0^{t_1}\dots\int_0^{t_{n-1}}\,dt_{n}\dots dt_1=\frac1{n!}.
\end{equation*}
Therefore from estimate~\eqref{e:last est} it follows that
\begin{equation*}
\lVert f(A)-p(A)\rVert\le\frac1{n!}
\max_{\substack{t\in[0,1]\\\mu\in\co\{z_1,z_{2},\dots,z_{n}\}}}\bigl\lVert\Omega(A)f^{{(n)}}
\bigl((1-t)\mu\mathbf1+tA\bigr)\bigr\rVert.\qed
\end{equation*}
\renewcommand\qed{}
\end{proof}

\begin{remark}\label{r:d co}
For numerical calculations, it may be useful to note that the maximum can be taken over the boundary $\partial\co$ of a convex hull instead of the whole convex hull:
\begin{multline*}
\max_{\mu\in\co\{z_1,z_{2},\dots,z_{n}\}}\bigl\lVert\Omega(A)f^{{(n)}}
\bigl((1-t)\mu\mathbf1+tA\bigr)\bigr\rVert\\
=\max_{\mu\in\partial\co\{z_1,z_{2},\dots,z_{n}\}}\bigl\lVert\Omega(A)f^{{(n)}}
\bigl((1-t)\mu\mathbf1+tA\bigr)\bigr\rVert.
\end{multline*}
Indeed, by the Hahn-Banach theorem,
\begin{equation*}
\bigl\lVert\Omega(A)f^{{(n)}}
\bigl((1-t)\mu\mathbf1+tA\bigr)\bigr\rVert=\max_{\lVert \xi\rVert\le1}
\xi\bigl[\Omega(A)f^{{(n)}}\bigl((1-t)\mu\mathbf1+tA\bigr)\bigr],
\end{equation*}
where the functional $\xi$ runs over the unit ball of the dual space of the space of all matrices. The function
\begin{equation*}
\xi\mapsto\xi\bigl[\Omega(A)f^{{(n)}}\bigl((1-t)\mu\mathbf1+tA\bigr)\bigr]
\end{equation*}
is analytic. Therefore, by the maximum modulus principle,
\begin{multline*}
\max_{\mu\in\co\{z_1,z_{2},\dots,z_{n}\}}\bigl|\xi\bigl[\Omega(A)f^{{(n)}}\bigl((1-t)\mu\mathbf1+tA\bigr)\bigr]
\bigr|\\
=\max_{\mu\in\partial\co\{z_1,z_{2},\dots,z_{n}\}}\xi\bigl[\Omega(A)f^{{(n)}}\bigl((1-t)\mu\mathbf1+tA\bigr)\bigr].
\end{multline*}
Taking maximum over all functionals $\xi$ of unit norm, we arrive at the equality being proved.
\end{remark}

Our Theorem~\ref{t:Mathias2} was inspired by the following result.

\begin{corollary}[{\rm\cite[Corollary 2]{Mathias93a},~\cite[Theorem 4.8]{Higham08}}]\label{c:Mathias}
Let the Taylor series
\begin{equation*}
f(\lambda)=\sum_{k=0}^\infty c_k(\lambda-z_1)^k,
\end{equation*}
where $c_k\in\mathbb C$,
converges on an open circle of radius $r$ with the center at $z_1$, and the spectrum of a square matrix $A$ is contained in this circle. Then
\begin{equation*}
\biggl\Vert f(A)-\sum_{k=0}^{n-1}c_k(A-z_1\mathbf1)^k\biggr\Vert\le
\frac1{n!}\max_{t\in[0,1]}\bigl\lVert (A-z_1\mathbf1)^n f^{(n)}\bigl((1-t)z_1\mathbf1+tA\bigr)\bigr\rVert.
\end{equation*}
\end{corollary}

In corollaries below, we simplify the estimate from Theorem~\ref{t:Mathias2} for the case of the most important function $f(z)=e^z$.

In notation of Theorem~\ref{t:Mathias2}, we set
\begin{align*}
\alpha&=\max\{\,\Real\lambda:\,\lambda\in\sigma(A)\,\}, \\
\beta&=\max_k\Real z_k,\\
\gamma&=\max\{\alpha,\beta\}.
\end{align*}

\begin{corollary}\label{c:Mathias3}
Let the assumptions of Theorem~\ref{t:Mathias2} be fulfilled and $f(z)=e^z$. Then
\begin{equation*}
\lVert e^A-p(A)\rVert\le\frac1{n!}
\max_{t\in[0,1]}e^{(1-t)\beta}\lVert\Omega(A) e^{tA}\rVert.
\end{equation*}
\end{corollary}
\begin{proof}
Clearly, $f^{(n)}(z)=e^z$. Therefore
\begin{equation*}
f^{{(n)}}\bigl((1-t)\mu\mathbf1+tA\bigr)=e^{(1-t)\mu}e^{tA}.
\end{equation*}
It remains to observe that
\begin{equation*}
\max_{\mu\in\co\{z_1,z_{2},\dots,z_{n}\}}|e^{(1-t)\mu}|=e^{(1-t)\beta}.\qed
\end{equation*}
\renewcommand\qed{}
\end{proof}

The following three corollaries are more effective (but rougher) versions of the previous one.

We denote by $\lVert\cdot\rVert_{2\to2}$ the matrix norm induced by the Euclidian norm on $\mathbb C^n$.

\begin{corollary}\label{c:Mathias4}
Let the assumptions of Theorem~\ref{t:Mathias2} be fulfilled and $f(z)=e^z$. Then
\begin{equation*}
\lVert e^A-p(A)\rVert_{2\to2}\le e^{\gamma}\frac{\lVert\Omega(A)\rVert_{2\to2}}{n!}
\sum_{j=0}^{n-1}\frac{(2\Vert A\Vert_{2\to2})^j}{j!}.
\end{equation*}
where the matrix $A$ has the size $n\times n$.
\end{corollary}
\begin{proof}
From Corollary~\ref{c:Mathias3} it follows that
\begin{equation*}
\lVert e^A-p(A)\rVert_{2\to2}\le\frac{\lVert\Omega(A)\rVert_{2\to2}}{n!}
\max_{t\in[0,1]}e^{(1-t)\beta}\lVert e^{tA}\rVert_{2\to2}.
\end{equation*}
Next, we make use of the estimate~\cite[p.~131, Lemma 10.2.1]{Bylov-Vinograd:rus-in-eng}, \cite[p.~68, formula (13)]{Gelfand-Shilov-GF3:eng}
\begin{equation*}
\Vert e^{At}\Vert_{2\to2}\le e^{\alpha t}\sum_{j=0}^{n-1}\frac{(2t\Vert A\Vert_{2\to2})^j}{j!},\qquad t\ge0.\qed
\end{equation*}
\renewcommand\qed{}
\end{proof}

\begin{corollary}\label{c:Mathias6}
Let the assumptions of Theorem~\ref{t:Mathias2} be fulfilled and $f(z)=e^z$.
Let the matrix $A$ be represented in the triangular Schur form~\cite{Golub-Van_Loan96:eng} $A=Q^{-1}BQ$, where $B$ is triangular and $Q$ is unitary. Further, let $B=D+N$, where $D$ is diagonal and $N$ is strictly triangular.
Then
\begin{equation*}
\lVert e^A-p(A)\rVert_{2\to2}\le e^{\gamma}\frac{\lVert\Omega(A)\rVert_{2\to2}}{n!}
\sum_{j=0}^{n-1}\frac{\lVert N\rVert_{2\to2}^j}{j!},
\end{equation*}
where the matrix $A$ has the size $n\times n$.
\end{corollary}
\begin{proof}
The proof is similar to that of Corollary~\ref{c:Mathias4} and based on the estimate~\cite{Van_Loan77}
\begin{equation*}
\lVert e^{At}\rVert=\lVert e^{Bt}\rVert\le e^{\alpha t}\sum_{k=0}^{n-1}\frac{\lVert Nt\rVert^k}{k!},\qquad t\ge0.\qed
\end{equation*}
\renewcommand\qed{}
\end{proof}

\begin{corollary}\label{c:Mathias5}
Let the assumptions of Theorem~\ref{t:Mathias2} be fulfilled and $f(z)=e^z$. Let the matrix $A$ be normal.
Then
\begin{equation*}
\lVert e^A-p(A)\rVert_{2\to2}\le e^{\gamma}\frac{\lVert\Omega(A)\rVert_{2\to2}}{n!}.
\end{equation*}
\end{corollary}
\begin{proof}
For the normal matrix $A$, we have $N=0$. Therefore the proof follows from Corollary~\ref{c:Mathias6}.
\end{proof}

Let us discuss whether the estimate is close to real accuracy.

\begin{example}\label{ex:Cayley-Hamilton theorem}
Let the points of interpolation $z_1$, \dots, $z_{n}$ be taken coinciding with the points of the spectrum of $A$ (counted according to their algebraic multiplicities). Then $\Omega$ is the characteristic polynomial of a matrix $A$. By the Cayley--Hamilton theorem, $\Omega(A)=0$. Thus, in this case, Theorem~\ref{t:Mathias2} implies the well-known identity $p(A)=f(A)$. Similarly, if $\Omega$ and its derivatives are small on the spectrum of $A$, the factor $\Omega(A)$ is also small.
\end{example}

\begin{example}\label{ex:Chebyshev}
Let $A$ be a Hermitian matrix with the spectrum lying in $[-1,1]$. Let the points of interpolation be the zeroes of the Chebyshev poly\-nomial of the first kind~\cite[\S~3.3]{Davis_PhJ} of degree $n$ on $[-1,1]$. In this case, $\Omega$ is this Chebyshev polynomial; if its leading coefficient is 1, then the maximal absolute value of $\Omega$ on $[-1,1]$ is $\frac1{2^{n-1}}$. Therefore, Corollary~\ref{c:Mathias5} implies that
\begin{equation*}
\lVert e^A-p(A)\rVert
\le\frac{1}{n!}\lVert\Omega(A)\rVert e^{\gamma}
\le\frac1{n!}\frac e{2^{n-1}}.
\end{equation*}
If the spectrum of $A$ is not known exactly, the sharp estimate (for this polynomial) is
\begin{equation*}
\lVert e^A-p(A)\rVert\le\max_{\lambda\in[0,1]}|e^\lambda-p(\lambda)|.
\end{equation*}
We compare these two estimates for $n=10$: we have $\frac1{n!}\frac e{2^{n-1}}=1.46\cdot10^{-9}$ and
$\max_{\lambda\in[0,1]}|e^\lambda-p(\lambda)|=0.60\cdot10^{-9}$. The comparison shows that the estimate from Corollary~\ref{c:Mathias5} is rather close to sharp one.
\end{example}

\section{Numerical experiment}\label{s:experiment}
Theorem~\ref{t:Mathias2} can help to estimate whether the accuracy of the approximation of the matrix function $f(A)$ by a matrix polynomial $p(A)$ is good enough for given points of interpolation. We give an example of such a verification based on Corollary~\ref{c:Mathias3}.

We put $N=1024$. We take complex numbers $d_i$, $i=1,\dots,N$, uniformly distributed in $[-1,0]\times [-i\pi,i\pi]$. We consider the diagonal matrix $D$ of the size $N\times N$ with the diagonal entries $d_i$. We take a matrix $T$, whose entries are random numbers uniformly distributed in $[-1,1]$. Then, we consider the matrix $A=TDT^{-1}$. Clearly, $\sigma(A)$ consists of the numbers $d_i$. We interpret $A$ as a random matrix whose spectrum is contained in the rectangle $[-1,0]\times [-i\pi,i\pi]$. In Fig.~\ref{f:cir} we show an example of the spectrum of such a matrix.

For $t\in[0,1]$, we take as the sharp matrix $e^{tA}$ the matrix $Te^{tD}T^{-1}$.

We take the following 16 points as interpolation points:
\begin{equation*}
0,\pm i\pi, \pm i\pi/2, \pm3 i\pi/4; -1, -1\pm i\pi, -1\pm i\pi/2, -1\pm3 i\pi/4; -1/2\pm i\pi.
\end{equation*}
They are marked in Fig.~\ref{f:cir} by the sign $\otimes$.
These points are chosen heuristically. We calculate the interpolation polynomial $p$ and the polynomial $\Omega$, which correspond to these points, and substitute the matrix $A$ into them.

Next we calculate $\lVert\Omega(A)e^{tA}\rVert_{2\to 2}=\lVert\Omega(A)Te^{tD}T^{-1}\rVert_{2\to 2}$ for $t=0.01k$, where $k=0,\dots,100$, and take the maximum of these numbers as an approximate value of $\max_{t\in[0,1]}e^{(1-t)\beta}\lVert\Omega(A) e^{tA}\rVert$ (we put $\beta=0$). Finally, we divide the result by $16!$ and, thus, obtain the estimate from Corollary~\ref{c:Mathias3}; we denote it by $e_1$. We also calculate the true accuracy $e_0=\lVert e^{A}-p(A)\rVert_{2\to 2}$ and the condition number $\varkappa(T)=\lVert T\rVert_{2\to 2}\cdot\lVert T^{-1}\rVert_{2\to 2}$.

We repeated the described experiment 100 times. After that, we excluded 3 results when $\varkappa(T)>10^{5}$. Finally, we calculated the average values. They are as follows:
average $e_0$ is $9.36\cdot10^{-6}$ with the standard deviation $1.37\cdot10^{-5}$,
average $e_1$ is $2.57\cdot10^{-5}$ with the standard deviation $3.33\cdot10^{-5}$,
average $e_1/e_0$ is $3.03$ with the standard deviation $1.11$,
average $\varkappa(T)$ is $1.84\cdot10^5$ with the standard deviation $2.44\cdot10^5$.

The average value $3.03$ of $e_1/e_0$ shows that the estimate is rather close to the true value. So, we can assume that for not very bad matrices $A$ of the size $N\times N$ with the spectrum in the rectangle $[-1,0]\times [-i\pi,i\pi]$, the interpolation polynomial with the considered interpolation points usually approaches $e^A$ with accuracy about $3\cdot10^{-5}$.

\begin{figure}[htb]
\begin{center}
\includegraphics[width=\textwidth]{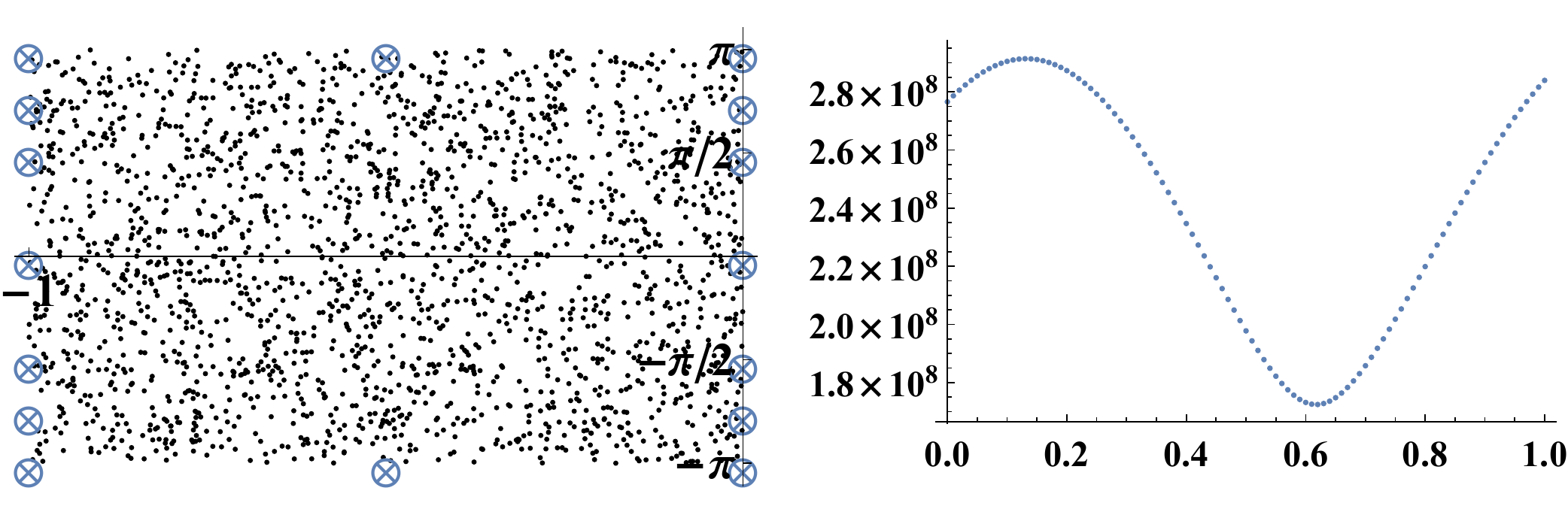}
\caption{Left: the spectrum of $A$ and the points of interpolation; right: the norms $\lVert\Omega(A)e^{tA}\rVert$ for $t=0.01k$, $k=0,\dots,100$}\label{f:cir}
\end{center}
\end{figure}

\section*{Acknowledgements}\label{s:Acknowledgements}
The first author was supported by the Ministry of Education and Science of the Russian Federation under state order No.~3.1761.2017/4.6.
The second author was supported by the Russian Foundation for Basic Research under research project No. 19-01-00732 А.

\providecommand{\bysame}{\leavevmode\hbox to3em{\hrulefill}\thinspace}
\providecommand{\MR}{\relax\ifhmode\unskip\space\fi MR }
\providecommand{\MRhref}[2]{%
  \href{http://www.ams.org/mathscinet-getitem?mr=#1}{#2}
}
\providecommand{\href}[2]{#2}


\begin{thebibliography}{10}

\bibitem{Bylov-Vinograd:rus-in-eng}
B.~F. Bylov, R.~\`E. Vinograd, D.~M. Grobman, and V.~V. Nemyckii, \emph{The
  theory of {L}yapunov exponents and its applications to problems of stability},
  Izdat. ``Nauka'', Moscow, 1966, (in Russian). \MR{0206415}

\bibitem{Crouzeix04}
M. Crouzeix, \emph{Bounds for analytical functions of matrices}, Integral
  Equations Operator Theory \textbf{48} (2004), no.~4, 461--477. \MR{2047592}

\bibitem{Davis_PhJ}
Ph.~J. Davis, \emph{Interpolation and approximation}, Dover Publications, Inc.,
  New York, 1975. \MR{0380189}

\bibitem{Frommer-Simoncini08a}
A.~Frommer and V.~Simoncini, \emph{Matrix functions}, in Model order
  reduction: theory, research aspects and applications, W.~H. Schilders, H.~A.
  van~der Vorst, and J.~Rommes, eds., vol.~13 of Mathematics in Industry,
  Springer, Berlin, 2008, pp.~275--303. \MR{2497756}

\bibitem{Gelfand-Shilov-GF3:eng}
I.~M. Gel$'$fand and G.~E. Shilov, \emph{Generalized functions. {V}ol. 3:
  {T}heory of differential equations}, Academic Press, New York--London, 1967,
  Translated from the Russian. \MR{0217416}

\bibitem{Gelfond:eng1}
A.~O. Gel$'$fond, \emph{Calculus of finite differences}, second ed., GIFML,
  Moscow, 1959, (in Russian); translated by Hindustan Publishing Corp., Delhi,
  in series International Monographs on Advanced Mathematics and Physics, 1971.
  \MR{0342890}

\bibitem{Gil-LMA-93}
M.~I. Gil$'$, \emph{Estimate for the norm of matrix-valued functions}, Linear and
  Multilinear Algebra \textbf{35} (1993), no.~1, 65--73. \MR{1310964}

\bibitem{Gil-MPAG-08}
\bysame, \emph{Estimates for entries of matrix valued functions of infinite
  matrices}, Math. Phys. Anal. Geom. \textbf{11} (2008), no.~2, 175--186.
  \MR{2438734}

\bibitem{Gil-ASM-14}
\bysame, \emph{A norm estimate for holomorphic operator functions in an ordered
  {B}anach space}, Acta Sci. Math. (Szeged) \textbf{80} (2014), no.~1-2,
  141--148. \MR{3236255}

\bibitem{Golub-Van_Loan96:eng}
G.~H. Golub and Ch.~F. Van~Loan, \emph{Matrix computations}, third ed., Johns
  Hopkins Studies in the Mathematical Sciences, Johns Hopkins University Press,
  Baltimore, MD, 1996. \MR{1417720}

\bibitem{Greenbaum04}
A.~Greenbaum, \emph{Some theoretical results derived from polynomial numerical
  hulls of {J}ordan blocks}, Electron. Trans. Numer. Anal. \textbf{18} (2004),
  81--90. \MR{2114450}

\bibitem{Higham08}
N.~J. Higham, \emph{Functions of matrices: theory and computation}, Society for
  Industrial and Applied Mathematics (SIAM), Philadelphia, PA, 2008.
  \MR{2396439}

\bibitem{Hille-Phillips:eng}
E.~Hille and R.~S. Phillips, \emph{Functional analysis and semi-groups},
  American Mathematical Society Colloquium Publications, vol.~31, Amer. Math.
  Soc., Providence, Rhode Island, 1957. \MR{0089373}

\bibitem{Jordan}
Ch. Jordan, \emph{Calculus of finite differences}, third ed., Chelsea
  Publishing Co., New~York, 1965. \MR{0183987}

\bibitem{Kagstrom77}
Bo~K\'{a}gstr\"{o}m, \emph{Bounds and perturbation bounds for the matrix
  exponential}, Nordisk Tidskr. Informationsbehandling (BIT) \textbf{17}
  (1977), no.~1, 39--57. \MR{0440896}

\bibitem{Mathias93a}
R.~Mathias, \emph{Approximation of matrix-valued functions}, SIAM J. Matrix
  Anal. Appl. \textbf{14} (1993), no.~4, 1061--1063. \MR{1238920}

\bibitem{Van_Loan75:SME}
Ch.~F. Van~Loan, \emph{A study of the matrix exponential}, Numerical Analysis
  Report No. 10, University of Manchester, Manchester, UK, August 1975,
  Reissued as MIMS EPrint 2006.397, Manchester Institute for Mathematical
  Sciences, The University of Manchester, UK, November 2006.

\bibitem{Van_Loan77}
\bysame, \emph{The sensitivity of the matrix exponential}, SIAM J. Numer. Anal.
  \textbf{14} (1977), no.~6, 971--981. \MR{0468137}

\bibitem{Young81}
N.~J. Young, \emph{A bound for norms of functions of matrices}, Linear Algebra
  Appl. \textbf{37} (1981), 181--186. \MR{636219}

\end{thebibliography}
\end{document}